\documentclass[12pt]{article}
\usepackage{amsmath,amssymb,amsthm, amsfonts}
\usepackage{hyperref}
\usepackage{graphicx}
\usepackage{url}
\usepackage[mathlines]{lineno}
\usepackage{dsfont} 
\usepackage{color}
\definecolor{red}{rgb}{1,0,0}

\definecolor{blue}{rgb}{0,0,1}

\definecolor{green}{rgb}{0,.6,0}

\usepackage{float}

\usepackage{tikz}

\newtheorem{thm}{Theorem}[section]
\newtheorem{cor}[thm]{Corollary}
\newtheorem{lem}[thm]{Lemma}
\newtheorem{prop}[thm]{Proposition}

\newtheorem{obs}[thm]{Observation}

\theoremstyle{definition}
\newtheorem{rem}[thm]{Remark}

\theoremstyle{definition}
\newtheorem{defn}[thm]{Definition}

\theoremstyle{definition}
\newtheorem{ex}[thm]{Example}

\newcommand{\Z}{\operatorname{Z}}

\newcommand{\bit}{\begin{itemize}}
\newcommand{\eit}{\end{itemize}}
\newcommand{\ben}{\begin{enumerate}}
\newcommand{\een}{\end{enumerate}}
\newcommand{\beq}{\begin{equation}}
\newcommand{\eeq}{\end{equation}}
\newcommand{\bea}{\begin{eqnarray}} 
\newcommand{\eea}{\end{eqnarray}}
\newcommand{\bpf}{\begin{proof}}
\newcommand{\epf}{\end{proof}\ms}
\newcommand{\bmt}{\begin{bmatrix}}
\newcommand{\emt}{\end{bmatrix}}
\newcommand{\ms}{\medskip}

\newcommand{\lc}{\left\lceil}
\newcommand{\rc}{\right\rceil}

\newcommand{\noi}{\noindent}
\newcommand{\ceil}[1]{\lc #1 \rc}

\newcommand{\beqs}{\begin{equation*}} 
\newcommand{\eeqs}{\end{equation*}}
\newcommand{\beas}{\begin{eqnarray*}}
\newcommand{\eeas}{\end{eqnarray*}}

\newcommand{\zf}{\operatorname{\lfloor \operatorname{Z} \rfloor}}
\newcommand{\ptzf}{\operatorname{pt_{\zf}}}
\newcommand{\ptz}{\operatorname{pt_{\Z}}}
\newcommand{\thz}{\operatorname{th_{\Z}}}
\newcommand{\pt}{\operatorname{pt}}
\newcommand{\throt}{\operatorname{th}}
\newcommand{\thzf}{\operatorname{th_{\zf}}}
\newcommand{\up}[1]{^{(#1)}}
\newcommand{\floor}[1]{\lfloor #1 \rfloor}
\newcommand{\calf}{\mathcal{F}}

\title{Throttling for Zero Forcing and Variants}
\author{Joshua Carlson\thanks{Department of Mathematics, Iowa State University, Ames, IA 50011, USA, jmsdg7@iastate.edu.}}
\date{\today}

\begin{document}
\maketitle

\begin{abstract} 
Zero forcing is a process on a graph in which the goal is to force all vertices to become blue by applying a color change rule. Throttling minimizes the sum of the number of vertices that are initially blue and the number of time steps needed to color every vertex. This paper provides a new general definition of throttling for variants of zero forcing and studies throttling for the minor monotone floor of zero forcing. The technique of using a zero forcing process to extend a given graph is introduced. For standard zero forcing and its floor, these extensions are used to characterize graphs with throttling number $\leq t$ as certain minors of cartesian products of complete graphs and paths. Finally, these characterizations are applied to determine graphs with extreme throttling numbers.
\end{abstract}

\noi {\bf Keywords} Zero forcing, propagation time, throttling, minor monotone floor

\noi{\bf AMS subject classification} 05C57, 05C15, 05C50

\begin{section}{Introduction}\label{intro}
Zero forcing is a process on graphs in which an initial set of vertices is colored blue (with the remaining vertices colored white) and vertices can force white vertices to become blue according to a color change rule. When using the color change rule, the goal is to eventually color every vertex in graph. Zero forcing can be used to model graph searching \cite{searching}, the spread of information on graphs \cite{BY13Throt}, and control of quantum systems \cite{QuantControl,QuantProp}.  Naturally, it is useful to know the smallest possible size of an initial set that can be used to color all vertices in the graph blue. It is also useful to know the time it takes to complete this process (often called propagation time). The idea of throttling is to study the relationship between the size of the initial set and its propagation time. Richard Brualdi posed the problem of minimizing the sum of these two quantities in 2011 (see \cite{BY13Throt}).

Unless otherwise stated, the graphs in this paper are simple, undirected, and finite. For a graph $G$, $V(G)$ and $E(G)$ denote the sets of vertices and edges of $G$ respectively. The cardinality of $V(G)$ is often denoted as $|G|$. The \emph{(standard) color change rule} is that a blue vertex $u$ can force a white vertex $w$ to become blue if $w$ is the only white neighbor of $u$. In this case, it is said that $u$ forces $w$ which is denoted as $u \rightarrow w$. A vertex is \emph{active} if it is blue and has not yet performed a force. Note that in standard zero forcing, any vertex that performs a force becomes inactive and cannot perform another force. Let $G$ be a graph with $B \subseteq V(G)$ colored blue and $V(G) \setminus B$ colored white. If every vertex in $V(G)$ can be forced to become blue by repeatedly applying the standard color change rule, then $B$ is a \emph{(standard) zero forcing set} of $G$. The \emph{(standard) zero forcing number}, $\Z(G)$, is the minimum size of a standard zero forcing set of $G$. In \cite{AIM}, it is shown that the zero forcing number can be used to bound the minimum rank of a matrix associated with a graph. 

Zero forcing propagation is studied in \cite{Propagation}. The idea is to simultaneously perform all possible forces at each time step. Define $B\up{0} = B$ and for each $t \geq 0$, define $B\up{t+1}$ to be the set of vertices $w$ for which there exists a vertex $b \in \bigcup_{s=0}^t B\up{s}$ such that $w$ is the only neighbor of $b$ not in $\bigcup_{s=0}^t B\up{s}$. The \emph{(standard) propagation time of $B$ in $G$}, denoted $\pt(G, B)$, is the smallest integer $t'$ such that $V(G) = \bigcup_{t=0}^{t'} B\up{t}$. Propagation time is particularly important in the control of quantum systems (see \cite{QuantProp}). 

Throttling for standard zero forcing was first studied by Butler and Young in \cite{BY13Throt}. If $B$ is a zero forcing set of a graph $G$, the \emph{throttling number of $B$ in $G$} is $\throt(G, B) = |B| + \pt(G, B)$. The \emph{(standard) throttling number of $G$} is the minimum value of $\throt(G, B)$ where $B$ ranges over all zero forcing sets of $G$. For a given graph $G$ and an integer $k$, the \textsc{Zero Forcing Throttling} problem is to determine if the standard throttling number of $G$ is less than $k$. The many variations of zero forcing (see \cite{Parameters}) lead to many variations of throttling. In \cite{powerdomthrot}, it was shown that \textsc{Zero Forcing Throttling} and other variants are NP-Complete.

Commonly studied variants of zero forcing include positive semidefinite zero forcing and loop zero forcing (see \cite{Parameters}). Let $G$ be a graph. A connected component of $G$ is a maximally connected subgraph of $G$. Suppose $B$ is a set of blue vertices in $G$ and $G-B$ has $k$ separate connected components. Let $W_1, \ldots , W_k$ be the sets of (white) vertices of the connected components of $G-B$. The \emph{positive semidefinite color change rule} applies the standard color change rule in $G[W_i \cup B]$ for any $1 \leq i \leq k$. The \emph{positive semidefinite zero forcing number} of a graph $G$ is denoted $\Z_+(G)$ and the positive semidefinite throttling number (studied in \cite{PSD}) is defined analogously to standard throttling. Loop zero forcing (see \cite{Parameters}) arises by considering a graph where every vertex has a loop. The \emph{loop color change rule} for simple graphs is to apply the standard color change rule, or if every neighbor of a white vertex $w$ is blue, then $w$ can force itself to become blue. The \emph{loop zero forcing number} of a graph $G$ is denoted $\Z_{\ell}(G)$.

If $G$ and $H$ are graphs and $G$ is a subgraph of $H$, write $G \leq H$. If $G \leq H$ and $|V(G)| = |V(H)|$, $G$ is a \emph{spanning subgraph} of $H$ and $H$ is a \emph{spanning supergraph} of $G$. If $G$ is a minor of $H$, write $G \preceq H$. Note that this paper breaks the convention of using $H$ to denote a minor or subgraph of a graph $G$ because it considers many graph parameters that depend on majors or supergraphs of a given graph. For example, suppose $p$ is a graph parameter whose range is well-ordered. The \emph{minor monotone floor} of $p$ is defined as $\floor{p}(G) = \min\{p(H) \ | \ G \preceq H\}$. In \cite{Parameters}, it was shown that $\zf$, $\floor{\Z_+}$, and $\floor{\Z_{\ell}}$ are zero forcing parameters with their own unique color change rules. In particular, the \emph{$\zf$ color change rule} is to either apply the standard color change rule, or alternatively if a vertex $v$ is active and all neighbors of $v$ are blue, then $v$ can force any single white vertex $w$ to become blue. The latter condition of the $\zf$ color change rule is called ``hopping". If this condition is used, then it is said that $v$ forces $w$ by a hop. It was also shown in \cite{Parameters} that the minor monotone floors of various zero forcing parameters are related to tree-width, path-width, and proper path-width. In addition, the concepts of path-width and proper path-width were shown in \cite{ppwSearch} to have connections to search games on graphs. 

In Section \ref{UnivDefs}, a general definition of propagation and throttling is given that allows for the study of further variations. Throttling for $\zf$ is studied in Section \ref{sectTHZF} and an ``extension'' technique that can be used to characterize graphs with $\zf$ throttling number at most $t$ for a fixed positive integer $t$ is introduced. A similar characterization for standard throttling is given in Section \ref{sectStandThrot}. These characterizations are applied in Section \ref{sectApps} in order to quickly characterize graphs with extreme throttling numbers. Finally, in Section \ref{conclusion}, an observation is made about proving the complexity of $\zf$ throttling and possibilities for future work are given.

\end{section}

\begin{section}{General Propagation Time and Throttling}\label{UnivDefs}
This section gives new general definitions of propagation time and throttling for color change rules. Define an \emph{(abstract) color change rule} to be a set of conditions under which a vertex $u$ can force a white vertex $w$ to become blue in a graph whose vertices are colored white or blue. The notation $u \rightarrow w$ is used to indicate that vertex $u$ forced vertex $w$ to become blue. Let $G$ be a graph with $B \subseteq V(G)$ colored blue and $V(G) \setminus B$ colored white. Let $R$ be a given color change rule. Repeatedly apply $R$ to $G$ until it is no longer possible to do so and write down the forces $u \rightarrow w$ in the order in which they are performed. This list of forces is called a \emph{chronological list of $R$ forces of $B$} and the unordered set of forces that appear in the list is a \emph{set of $R$ forces of $B$}. Suppose $G$ is a graph and $\calf$ is a set of $R$ forces of $B \subseteq V(G)$. An \emph{$R$ forcing chain of $\calf$} is a sequence of vertices $(v_1, v_2, \ldots, v_k)$ in $G$ such that $(v_i \rightarrow v_{i+1}) \in \calf$ for each $1 \leq i \leq k-1$. An $R$ forcing chain of $\calf$ is \emph{maximal} if it is not properly contained in any other $R$ forcing chain of $\calf$. The set of vertices in $G$ that are blue after all forces in $\calf$ have been performed is an \emph{$R$ final coloring of $B$}.

\begin{rem}
Suppose $B'$ is an $R$ final coloring of a set $B \subseteq V(G)$ obtained by performing the forces in a chronological list of $R$ forces of $B$ (denoted by $\mathcal{L}$). Note that $B'$ consists of the vertices in $B$ together with all vertices that become forced in $\mathcal{L}$. Therefore, $B'$ does not depend on the chronological ordering of $\mathcal{L}$. This means that $R$ final colorings depend on sets of forces and not chronological lists of forces.
\end{rem}


Let $G$ be a graph and let $R$ be a given color change rule. An \emph{$R$ forcing set of $G$} is a set $B \subseteq V(G)$ of vertices such that $V(G)$ is an $R$ final coloring of $B$ for some set of $R$ forces. The \emph{$R$ forcing parameter}, $R(G)$, is the minimum size of an $R$ forcing set of $G$. An $R$ forcing set $B$ is a \emph{minimum $R$ forcing set of $G$} if $|B| = R(G)$. 

Note that the definition of standard propagation time of a set of vertices does not use sets of forces. This is because final colorings in standard zero forcing are unique and depend only on the initial set of blue vertices (see \cite{AIM}). However, there are variants of zero forcing that do not have unique final colorings (e.g., $\zf$ forcing). When performing a $\zf$ force by hopping, there are many choices for the white vertex that gets forced. Example 2.36 in \cite{Parameters} illustrates that it is possible to start with a blue $\zf$ forcing set $B$ and fail to color every vertex in the graph due to poor hopping choices. In this case, $B$ has at least two distinct sets of $\zf$ forces with different propagation times. This motivates the following definitions.

For a set of $R$ forces $\calf$ of $B \subseteq V(G)$, define $\calf\up{0} = B$ and for $t \geq 0$, $\calf\up{t+1}$ is the set of vertices $w$ such that the force $v \rightarrow w$ appears in $\calf$ and $w$ can be $R$ forced by $v$ if the vertices in $\bigcup_{i=0}^t \calf\up{i}$ are colored blue and the vertices in $V(G) \setminus \left(\bigcup_{i=0}^t \calf\up{i}\right)$ are colored white. The \emph{$R$ propagation time of $\calf$ in $G$}, denoted $\pt_R(G; \calf)$, is the least $t'$ such that $V(G) = \bigcup_{i=0}^{t'} \calf\up{i}$. If the $R$ final coloring induced by $\calf$ is not $V(G)$, then define $\pt_R(G;\calf) = \infty$. 
Note that $B$ is colored blue at time $0$, and for each $1 \leq t \leq \pt_R(G; \calf)$, time step $t$ takes place between time $t-1$ and time $t$ in $\calf$. A vertex in $G$ is \emph{active at time $t$} if it is blue at time $t$ and has not performed a force in time step $s$ for any $s \leq t$.

\begin{defn}\label{RPropSet}
Let $G$ be a graph with $B \subseteq V(G)$ and let $R$ be a given color change rule. The \emph{$R$ propagation time of $B$} is defined as 
\beas
\pt_R(G;B) = \min\{\pt_R(G;\calf) \ | \ \calf \text{ is set of $R$ forces of }B\}.
\eeas
\end{defn}

Note that Definition \ref{RPropSet} doesn't require the set $B$ to be an $R$ forcing set of $G$. This is because a set $\calf$ of $R$ forces that fails to color every vertex in $G$ has $\pt_R(G; \calf) = \infty$. Therefore, such a set $\calf$ does not realize $\pt_R(G; B)$ when $B$ is an $R$ forcing set of $G$. If $B$ is not an $R$ forcing set of $G$, then every set of $R$ forces of $B$ has infinite propagation time and $\pt_R(G;B) = \infty$. Another advantage of Definition \ref{RPropSet} is that it is not required to prove that a subset of vertices is an $R$ forcing set before discussing its propagation time. This is useful for proving Proposition \ref{ptProp} in the next section.

The (standard) propagation time of a graph (see \cite{Propagation}) considers the smallest propagation time among minimum zero forcing sets. The next definition generalizes this idea.

\begin{defn}\label{ptRgraph}
Let $G$ be a graph and let $R$ be a given color change rule. The \emph{$R$ propagation time of $G$} is defined as 
\beas
\pt_R(G) = \min\{\pt_R(G;B) \ | \ B \text{ is a minimum $R$ forcing set of }G\}.
\eeas
\end{defn}

\begin{defn}\label{RThrotSet}
Let $G$ be a graph with $B \subseteq V(G)$ and let $R$ be a given color change rule. The \emph{$R$ throttling number of $B$ in $G$} is 
\beas 
\throt_R(G;B) = |B| + \pt_R(G;B).
\eeas
\end{defn}

\begin{defn}\label{RThrotNum}
Let $G$ be a graph and let $R$ be a given color change rule. The \emph{$R$ throttling number} of $G$ is defined as 
\beas 
\throt_R(G) = \underset{B \subseteq V(G)}{\min}\{\throt_R(G;B)\}.
\eeas
\end{defn}

When comparing propagation time and throttling for various color change rules, $\Z$ is used to denote the standard zero forcing color change rule (i.e., $\ptz$ and $\thz$). 

\end{section}

\begin{section}{Throttling for the Minor Monotone Floor of Z.}\label{sectTHZF}
This section investigates propagation and throttling for the $\zf$ color change rule. Definition \ref{RPropSet} exhibits the connection between the $\zf$ propagation time of a subset $B \subseteq V(G)$ and the $\zf$ propagation time of a set of $\zf$ forces of $B$. The following proposition shows that the $\ptzf(G;B)$ can also be calculated by minimizing the standard zero forcing propagation time of $B$ on spanning supergraphs of $G$.

\begin{prop}\label{ptProp}
If $G$ is a graph and $B \subseteq V(G)$, then
\begin{eqnarray}
\pt_{\zf}(G;B) = \min\{\ptz(H;B) \ | \ G \leq H \text{ and }|G| = |H| \}. \label{eq1}
\end{eqnarray}
\end{prop}

\begin{proof}
Let $\calf$ be a set of $\zf$ forces of $B$ such that $\ptzf(G;B) = \ptzf(G; \calf)$. Note that every force in $\calf$ is either a $\Z$ force or a force by a hop. Let $G'$ be the graph obtained from $G$ by adding the edges $uw$ such that $u \rightarrow w$ appears in $\calf$ and $u \rightarrow w$ by a hop. Note that for each edge $uw \in E(G') \setminus E(G)$, $w$ is the only white neighbor of $u$ in $G'$ and $u$ is active at the time that $u \rightarrow w$ in $\calf$. This means that $u \rightarrow w$ is a valid $\Z$ force in $G'$ for each such edge. Thus, $\calf$ is a set of $\Z$ forces of $B$ in $G'$ and $\pt_{Z}(G'; \calf) = \ptzf(G; \calf)$. Therefore, 
\beas 
\ptzf(G;B) = \pt_{Z}(G'; \calf) \geq \min\{\ptz(H;B) \ | \ G \leq H \text{ and }|G| = |H| \} .
\eeas
Now let $H'$ be a spanning supergraph of $G$ such that the right hand side of \eqref{eq1} is equal to $\ptz(H', B)$. Let $\calf$ be a set of $\Z$ forces of $B$ such that $\ptz(H', \calf) = \ptz(H', B)$. Consider applying $\calf$ to $B$ in $G$ and hopping when an edge is missing. If $(u \rightarrow w) \in \calf$ and $uw \in E(H') \setminus E(G)$, then $u$ can $\zf$ force $w$ in $H' - uw$ by a hop when $u \rightarrow w$ in $\calf$. If $(u \rightarrow w) \in \calf$ and $uw \notin E(H') \setminus E(G)$, then $u$ will $\Z$ force $w$ in $G$ exactly the way $u \rightarrow w$ in $H'$. If $(u \rightarrow w) \notin \calf$, then the propagation time of $\calf$ does not change regardless of whether $uw$ is removed from $H'$ to obtain $G$. This means that $\calf$ is a set of $\zf$ forces of $B$ in $G$ with $\ptzf(G;\calf) = \ptz(H';\calf)$. Thus,
\beas 
\ptzf(G;B) \leq \ptzf(G;\calf) = \ptz(H', B) = \min \{\ptz(H;B) \ | \ G \leq H \text{ and }|G| = |H| \}.
\qedhere
\eeas 
\end{proof}

By the definition of minor monotone floor given in Section \ref{intro}, $\zf$ is minor monotone (i.e., $\zf(G) \leq \zf(H)$ if $G \preceq H$). Since any $\Z$ forcing set of a graph $G$ is also a $\zf$ forcing set of $G$, $\zf$ is bounded above by $\Z$. These facts together with Definitions \ref{ptRgraph}, \ref{RThrotSet}, and \ref{RThrotNum} can be used to extend the above proposition and give similar results for the $\zf$ propagation time of a graph and $\zf$ throttling.

\begin{cor}
Let $G$ be a graph. Then
\beas 
\pt_{\zf}(G) = \min\{\ptz(H) \ | \ G \leq H \text{ with } |G| = |H| \text{ and } \zf(G)=\Z(H) \}. 
\eeas
\end{cor}
\begin{proof}
Let $H$ be a spanning supergraph of $G$ with $B$ a standard zero forcing set of $H$. Then, $\zf(G) \leq \zf(H) \leq \Z(H) \leq |B|$. Therefore, assuming that $|B| = \zf(G)$ gives $|B| = \Z(H)$ which means that $B$ is a minimum zero forcing set of $H$. By Proposition \ref{ptProp}, it follows that 
\beas 
\ptzf(G) &=& \min\{\ptzf(G; B) \ | \ \zf(G) = |B|\}\\[.5 em]
&=& \min\{\min\{\ptz(H; B) \ | \ G \leq H \text{ and } |G| = |H|\} \ | \ \zf(G) = |B|\}\\[.5 em]
&=& \min\{\ptz(H; B) \ | \ G \leq H \text{ with } |G| = |H| \text{ and } \zf(G) = |B|\}\\[.5 em]
&=& \min\{\ptz(H) \ | \ G \leq H \text{ with } |G| = |H| \text{ and } \zf(G)=\Z(H)\}. \qedhere
\eeas 

\end{proof}

\begin{cor}
If $G$ is a graph and $B \subseteq V(G)$, then 
\beas 
\thzf(G; B) = \min\{\thz(H;B) \ | \ G \leq H \text{ and }|G| = |H| \}.
\eeas
\end{cor}

\begin{cor}\label{zfThrotSpanSupers}
Let $G$ be a graph. Then
\beas 
\thzf(G) = \min\{\thz(H) \ | \ G \leq H \text{ and }|G| = |H| \}.
\eeas
\end{cor}

\begin{thm}\label{zfThrotSubMon}
The $\zf$ throttling number is subgraph monotone. In particular, if $G$ and $H$ are graphs with $G \leq H$, then $\thzf(G) \leq \thzf(H)$.
\end{thm}
\begin{proof}
Let $H$ be a graph. By Corollary \ref{zfThrotSpanSupers}, $\thzf(G') \leq \thzf(H)$ for any spanning subgraph $G'$ of $H$. Let $v \in V(H)$ and let $E(v)$ be the set of all edges in $H$ incident with $v$. Define $G' = H - E(v)$. Note that $\thzf(G') \leq \thzf(H)$. Choose $B' \subseteq V(G')$ such that $\thzf(G'; B') = \thzf(G')$. Let $\calf'$ be a set of $\zf$ forces of $G'$ with $\ptzf(G'; \calf') = \ptzf(G'; B')$. The goal is to produce a set $B \subseteq V(G'-v)$ and a set of $\zf$ forces, $\calf$, of $B$ such that $|B| \leq |B'|$ and $\ptzf(G'-v, \calf) \leq \ptzf(G'; \calf')$. Let $v_1 \rightarrow v_2 \rightarrow \cdots \rightarrow v_k$ be the maximal $\zf$ forcing chain of $\calf'$ that contains $v$. If $k=1$, then it suffices to choose $B = B'\setminus\{v\}$ and $\calf = \calf'$. Now assume $k > 1$. Note that $v = v_i$ for some $1 \leq i \leq k$. Define $B$ and $\calf$ as 
\beas 
B = 
\begin{cases}
(B' \setminus \{v_i\})\cup \{v_{i+1}\} & \text{if } i = 1,\\
B' & \text{otherwise},
\end{cases}
\eeas
and
\beas 
\calf =
\begin{cases}
\calf' \setminus \{v_{i} \rightarrow v_{i+1}\} & \text{if } i = 1,\\
(\calf' \setminus \{v_{i-1} \rightarrow v_i, v_i \rightarrow v_{i+1}\}) \cup \{v_{i-1} \rightarrow v_{i+1}\} & \text{if } 1 < i < k,\\
\calf' \setminus \{v_{i-1} \rightarrow v_{i}\} & \text{if } i = k.
\end{cases}
\eeas

Recall that $v$ is an isolated vertex in $G'$. So when $1 < i < k$, $v_{i-1} \rightarrow v_i$ and $v_i \rightarrow v_{i+1}$ by hopping in $G'$. This means at the time that $v_{i-1} \rightarrow v_i$ in $G'$, $v_{i-1}$ can force $v_{i+1}$ by a hop in $G'-v$. In the other cases, simply remove the appropriate force from $\calf'$. So in all cases, it is clear that $|B| \leq |B'|$ and $\ptzf(G'-v; \calf) \leq \ptzf(G'; \calf')$. Also note that $G' - v = H - v$. Thus, for all $1 \leq i \leq k$,
\beas 
\thzf(H - v) &\leq& |B| + \ptzf(G'-v; \calf) \leq |B'| + \ptzf(G'; \calf') = \thzf(G') \leq \thzf(H).
\eeas
Since $v$ was chosen arbitrarily, it follows that removing vertices from $H$ will not increase the $\zf$ throttling number.
\end{proof}

Since $\zf$ is minor monotone, it is natural to ask if Theorem \ref{zfThrotSubMon} can be strengthened to say that $\thzf$ is minor monotone. This question is answered negatively (see Theorem \ref{thzfNotMinMon}) once a characterization of $\thzf$ is obtained. Note that Theorem \ref{zfThrotSubMon} can be extended in other ways. For each $p \in \{\Z_+, \Z_{\ell}\}$, the color change rule for $\floor{p}$ takes the color change rule for $p$ and allows hopping. This leads to the following corollary.

\begin{cor}\label{otherP}
Suppose $G$ is a graph and $B \subseteq V(G)$. Then for each $p \in \{Z_+, Z_{\ell}\}$, 
\beas 
\pt_{\floor{p}}(G;B) &=& \min\{\pt_p(H;B) \ | \ G \leq H \text{ and }|G| = |H| \},\\[10 pt]
\throt_{\floor{p}}(G;B) &=& \min\{\throt_p(H;B) \ | \ G \leq H \text{ and }|G| = |H| \},
\eeas 
and $\throt_{\floor{p}}$ is subgraph monotone.
\end{cor}

It is likely that Corollary \ref{otherP} will hold for any graph parameter $p$ such that $\floor{p}$ has a corresponding color change rule that takes the color change rule for $p$ and allows hopping. However, no other parameters $p$ have been shown to have this property. Note that if $B$ is a standard zero forcing set of a graph $G$, then $B$ is also a $\zf$ forcing set of $G$ with $\ptzf(G; B) \leq \ptz(G; B)$. Thus, it is immediate that for any graph $G$, $\thzf(G)$ is bounded above by $\thz(G)$. Butler and Young showed in \cite[page 66]{BY13Throt} that for any graph $G$ of order $n$, $\thz(G)$ is at least $\ceil{2\sqrt{n} - 1}$. By Corollary \ref{zfThrotSpanSupers}, this lower bound holds for $\thzf(G)$ as well. 
\begin{cor}\label{thzfLowerBd}
If $G$ is a graph of order $n$, then 
\beas 
\thzf(G) = \min\{\thz(H) \ | \ G \leq H \text{ and }|G| = |H| \} \geq \ceil{2\sqrt{n} - 1}. 
\eeas
\end{cor}
Since the $\zf$ throttling number is bounded above by the standard throttling number, any graph $G$ that achieves $\thz(G) = \ceil{2\sqrt{n} - 1}$ also achieves $\thzf(G) = \ceil{2\sqrt{n} - 1}$. It was shown in \cite{BY13Throt} that $\thz(P_n) = \ceil{2 \sqrt{n} - 1}$. Thus, it can be concluded that $\thzf(P_n) = \ceil{2\sqrt{n} - 1}$. The standard throttling number of a cycle was determined in \cite{PSD} as follows.

\begin{thm}\emph{\cite[Theorem 7.1]{PSD}}\label{standThrotCycle}
Let $C_n$ be a cycle on $n$ vertices. Define $m$ to be the largest integer such that $m^2 \leq n$ and $n = m^2 + r$. Then
\beas 
\thz(C_n) = 
\begin{cases}
2m-1 & \text{if } r=0 \text{ and } m \text{ is even},\\
2m & \text{if } 0 < r \leq m \text{ or }(r=0 \text{ and } m \text{ is odd}), \\
2m + 1 & \text{if } m < r < 2m+1.
\end{cases}
\eeas
\end{thm}
Theorem \ref{standThrotCycle} can be used to determine the $\zf$ throttling number of a cycle.
\begin{prop}\label{zfthrotCycle}
Let $C_n$ be a cycle on $n$ vertices. Then $\thzf(C_n) = \ceil{2 \sqrt{n} - 1}$.

\end{prop}
\begin{proof}
Define $m$ to be the largest integer such that $m^2 \leq n$ and $n = m^2 + r$. Note that if $m$ is even or $r > 0$, then the conditions in Theorem \ref{standThrotCycle} are equivalent to the conditions for $\thz(P_n)$ in \cite{BY13Throt}. So in this case, $\thzf(C_n) = \thz(P_n) = \ceil{2 \sqrt{n} - 1}$. Now suppose $m$ is odd and $r = 0$. So $n = m^2$ and $\thz(C_n) = 2m = \ceil{2 \sqrt{n} - 1} + 1$. In this case, construct a $\zf$ forcing set $B$ with $|B| = m$ and $\ptzf(C_n; B) \leq m-1$ as follows. Draw $C_n$ by arranging the vertices in an $m$ by $m$ array and adding the edges as in Figure \ref{CycleSnake}. Let $B$ be the set of vertices in the left column. Note that in each time step, every active vertex can force the vertex to its right to become blue (sometimes by a hop), so every vertex becomes blue one column at a time. Let $\calf$ be the set of $\zf$ forces of $B$ obtained by this process. Clearly $|B| = m$ and $\ptzf(C_n; B) \leq \ptzf(C_n; \calf) = m-1$. Thus $\thzf(C_n) \leq 2m - 1 = \ceil{2\sqrt{n} - 1}$.
\end{proof}

\begin{figure}[H] \begin{center}
\scalebox{1}{\includegraphics{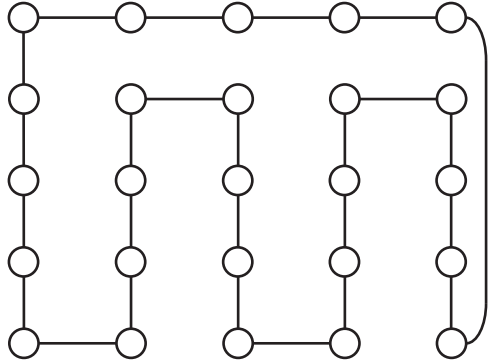}}\\
\caption{The cycle $C_n$ with $n = m^2$ and $m=5$.}\label{CycleSnake}
\end{center}
\end{figure}
Example \ref{StarWheelEx} uses Theorem \ref{zfThrotSubMon} to demonstrate that if $\thz(G) > \ceil{2 \sqrt{n} - 1}$, then $\thzf(G)$ can differ greatly from $\thz(G)$.  

\begin{figure}[H] \begin{center}
\scalebox{1}{\includegraphics{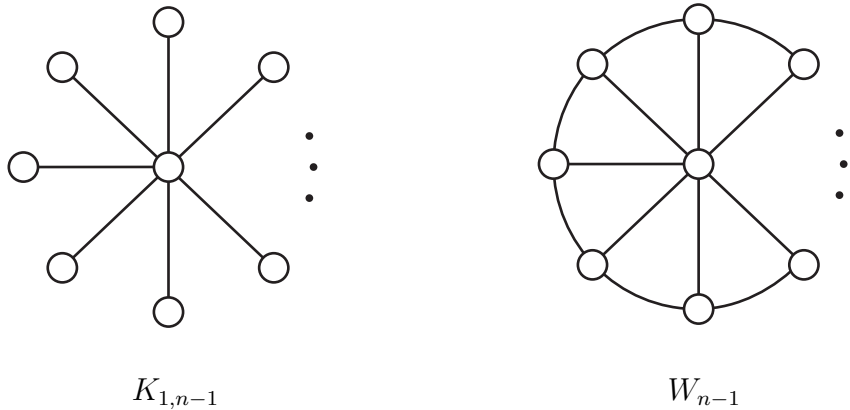}}\\
\end{center}
\hspace{1.63 in} $K_{1,n-1}$ \hspace{2.24 in} $W_{n-1}$
\begin{center}
\caption{The star on $n$ vertices alongside the wheel as a spanning supergraph.}\label{starwheel}
\end{center}
\end{figure}
\begin{ex}\label{StarWheelEx}
Let $G$ be the star $K_{1, n-1}$ on $n$ vertices as shown on the left in Figure \ref{starwheel}. Since $\Z(G) = n-2$, it can be verified by inspection that $\thz(G) = n$. Consider the wheel $W_{n-1}$ on $n$ vertices as a spanning supergraph of $G$ (shown on the right of Figure \ref{starwheel}). Obtain $B \subseteq V(W_{n-1})$ by choosing the center vertex of the wheel and a set of vertices on the outside cycle that achieves optimal $\zf$ throttling for a cycle of order $n-1$. By Theorem \ref{zfThrotSubMon}, $\thzf(G) \leq \thzf(W_{n-1}) \leq \thzf(C_{n-1}) + 1 \leq \ceil{2\sqrt{n-1} - 1} + 1$. Recall that $\thzf(G) \geq \ceil{2 \sqrt{n} - 1}$. Note that there are infinitely many integers $n$ such that $\ceil{2\sqrt{n-1} - 1} + 1 = \ceil{2\sqrt{n} - 1}$. So in these cases, $\thzf(G) = \ceil{2\sqrt{n} - 1}$.
\end{ex}

The \emph{largeur d'arborescence} of a graph was defined by Colin de Verdi\`ere in \cite{la} to measure the width of trees. Note that largeur d`arborescence is french for tree width. The \emph{largeur de chemin} of $G$, denoted by $\operatorname{lc}(G)$, was introduced in \cite{Parameters} as the analog of largeur d'arborescence that measures the width of paths. Formally, $\operatorname{lc}(G)$ is defined as the minimum $k$ for which $G$ is a minor of the Cartesian product $K_k \square P$ of a complete graph on $k$ vertices with a path. The \emph{proper path width} of a graph $G$, $\operatorname{ppw}(G)$, is the smallest $k$ such that $G$ is a partial linear $k$-tree (see \cite{Parameters}). These parameters are connected to $\zf$ by the following theorem.

\begin{thm}\emph{\cite[Theorems 2.18 and 2.39]{Parameters}}\label{LCZFCCR} For every graph $G$ having at least one edge, $\operatorname{lc}(G) = \operatorname{ppw}(G)=\zf(G)$.
\end{thm}

It is known that proper path-width is equivalent to the mixed search number of a graph (see \cite{ppwSearch}). Since $\operatorname{ppw}(G) = \zf(G) \leq \thzf(G)$ for any graph $G$, Theorem \ref{LCZFCCR} connects $\zf$ throttling to mixed searching. Theorem \ref{LCZFCCR} also exhibits a relationship between $\zf$ and graphs of the form $K_k \square P$. It is useful to capitalize on this relationship in order to characterize $\thzf(G)$. For a given a graph $G$, the idea is to extend $G$ by using a set of forces in $G$. The next definition constructs a graph from a given graph $G$, a standard zero forcing set $B \subseteq V(G)$, and a set of standard forces $\calf$. This construction is illustrated in Figure \ref{extExample}.
\begin{defn}\label{ext}
Let $G$ be a graph and let $B \subseteq V(G)$ be a standard zero forcing set of $G$. Suppose $\calf$ is a set of $\Z$ forces of $B$ with $\ptz(G; B) = \ptz(G; \calf)$. Let $P_1, P_2, \ldots , P_{|B|}$ be the induced paths in $G$ formed by the maximal forcing chains of $\calf$. For each vertex $v \in V(G)$, consider the path $P_i$ that contains $v$ and let $\tau(v)$ be the number of times in the propagation process of $\calf$ at which $v$ is active (possibly including time $0$). Define the \emph{(zero forcing) extension of $G$ with respect to $B$ and $\calf$}, denoted $\mathcal{E}(G, B, \calf)$, to be the graph obtained by the following procedure. 
\begin{enumerate}
\item From each path $P_i$ in $G$, construct a new path $P_i'$ so that for each $v \in P_i$, there are $\tau(v)$ copies of $v$ in $P_i'$, and for each pair $v_a$, $v_b \in P_i$ such that $v_a$ is forced before $v_b$ in $P_i$, every copy of $v_a$ is to the left of every copy of $v_b$ in $P_i'$.  Note that for each $1 \leq i \leq |B|$, $|V(P_i')| = \ptz(G; B) + 1$ and the paths $\{P_1', P_2', \ldots, P_{|B|}'\}$ can be arranged into a $|B|$ by $\pt(G; B) + 1$ array of vertices.
\item For each edge $uv \in E(G) \setminus \bigcup_{i=1}^{|B|}E(P_i)$, suppose $P_q$ and $P_r$ are the paths that contain $u$ and $v$ respectively. Since $u$ and $v$ must both be active before $u$ or $v$ can perform a force in $G$, there is at least one column in the $|B|$ by $\pt(G; B) + 1$ array such that a copy of $u$ and a copy of $v$ appear in that column. Draw an edge connecting the copy of $u$ in $P_q'$ and the copy of $v$ in $P_r'$ that are in the least such column.
\end{enumerate}
\end{defn}

\begin{ex}
Let $G$ be the graph shown on the left in Figure \ref{extExample}. Choose $B = \{v_1, v_4, v_7\}$ and let $\calf$ be the set of standard forces $\calf = \{v_1 \rightarrow v_2, v_2 \rightarrow v_3, v_4 \rightarrow v_5, v_5 \rightarrow v_6, v_7 \rightarrow v_8, v_8 \rightarrow v_9\}.$ Note that the forces in $\calf$ correspond to the horizontal edges in $G$ as shown in Figure \ref{extExample}. The {\color{red} numbers} above the vertices of $G$ indicate the time step in $\calf$ when that vertex is forced (making that vertex active at the next time in the propagation process). For example, $v_7 \rightarrow v_8$ in time step $1$ and $v_8 \rightarrow v_9$ in time step $3$. Since there are two times in $\calf$ at which $v_8$ active, there are two copies of $v_8$ in $\mathcal{E}(G; B; \calf)$, which is shown on the right in Figure \ref{extExample}.
\end{ex}
\begin{figure}[H] \begin{center}
\scalebox{1.1}{\includegraphics{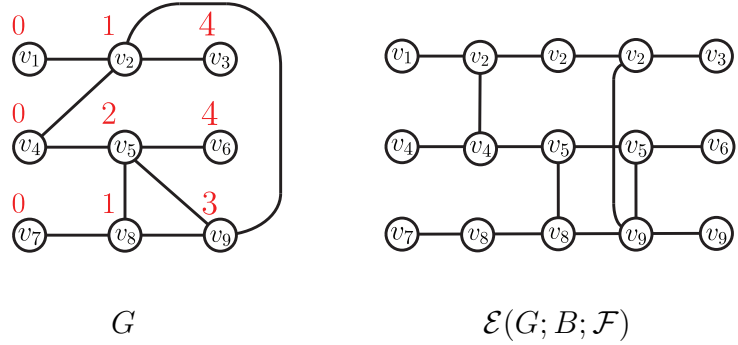}}\\
\end{center}
\hspace{1.69 in} $G$ \hspace{1.71 in} $\mathcal{E}(G; B; \calf)$
\begin{center}
\caption{$G$, $B$, and $\calf$ are illustrated alongside the extension $\mathcal{E}(G;B;\calf)$.}\label{extExample}
\end{center}
\end{figure}

Consider the graph $G = K_a \square P_b$. Define the $\emph{path edges}$ of $G$ to be the edges in each copy of $P_b$ in the Cartesian product. Likewise, define the $\emph{complete edges}$ of $G$ to be the edges in each copy of $K_a$ in the Cartesian product. For example, if $G$ is drawn so that $V(G)$ is arranged as an $a$ by $b$ array where each column induces a $K_a$ and each row induces a $P_b$, then the path edges of $G$ are the horizontal edges and the complete edges of $G$ are the vertical edges. Given a graph $G$, an edge $e \in E(G)$, a standard zero forcing set $B \subseteq V(G)$, and a set $\calf$ of standard forces in $G$ that uses $e$ to perform a force, the following definition constructs a standard zero forcing set in $G/e$ and a set of standard forces in $G/e$ that mimic $B$ and $\calf$ respectively.
\begin{defn}
Let $G$ be a graph with standard zero forcing set $B \subseteq V(G)$ and suppose $\calf$ is a set of forces of $B$. Let $e \in E(G)$ be an edge that is used to perform a force in $\calf$. Define $v_1 \rightarrow v_2 \rightarrow \cdots \rightarrow v_k$ to be the maximal forcing chain of $\calf$ that contains $e$. Note that $k \geq 2$. For each $1 \leq j \leq k-1$, let $e_j$ be the edge $v_jv_{j+1}$ and let $\vec{e}_j$ denote the force $v_j \rightarrow v_{j+1}$. So $e = e_i$ for some $1 \leq i \leq k-1$.  Define $v_e$ to be the vertex in $G/e$ obtained by contracting $e$ in $G$ and define the sets $B/e$ and $\calf/e$ as follows.
\beas 
B/e = 
\begin{cases}
(B \setminus \{v_i\}) \cup \{v_e\} & \text{if } i = 1,\\
B & \text{if } i > 1,\\

\end{cases}
\eeas
and
\beas 
\calf/e = 
\begin{cases}
(\calf \setminus \{\vec{e}_{i-1}, \vec{e}_i, \vec{e}_{i+1}\}) \cup \{v_{i-1} \rightarrow v_e, v_e \rightarrow v_{i+2}\} & \text{if } k > 2 \text{ and } 1 < i < k-1,\\
(\calf \setminus \{\vec{e}_i, \vec{e}_{i+1}\}) \cup \{v_e \rightarrow v_{i+2}\} & \text{if } k > 2 \text{ and } i = 1,\\
(\calf \setminus \{\vec{e}_{i-1}, \vec{e}_i\}) \cup \{v_{i-1} \rightarrow v_e\} & \text{if } k > 2 \text{ and } i = k-1,\\
\calf \setminus \{\vec{e}_i\} & \text{if } k = 2.
\end{cases}
\eeas
\end{defn}

Lemma \ref{contractZF} is used to prove Theorem \ref{thzfChar} which exhibits a relationship between $\thzf$ and graphs of the form $K_a \square P_{b+1}$.

\begin{lem}\label{contractZF}
Let $G$ be a graph. Suppose $B \subseteq V(G)$ is a standard zero forcing set of $G$ with a set of standard forces $\calf$. If $e=uv$ is an edge in $E(G)$ and $(u \rightarrow v) \in \calf$, then $\calf/e$ is a set of standard forces of $B/e$ in $G/e$ such that $\ptz(G/e, \calf/e) \leq \ptz(G; \calf)$. Furthermore, if $\calf$ and $B$ satisfy $\ptz(G; \calf) = \ptz(G; B)$ and $\thz(G) = \thz(G; B)$, then $\thz(G/e) \leq \thz(G)$.
\end{lem}

\begin{proof}
Let $G$ be a graph with standard zero forcing set $B \subseteq V(G)$. Let $\calf$ be a set of forces of $B$ and suppose $e = uv \in E(G)$ is an edge that is used to perform a force in $\calf$. Assume without loss of generality that $(u \rightarrow v) \in \calf$. Proceed by induction on $\ptz(G; \calf)$. If $\ptz(G; \calf) = 0$, then $B = V(G)$ and no such edge $e$ exists and there is nothing to prove. Suppose $\ptz(G; \calf) = 1$. In this case, it is clear that $\calf/e$ is a set of forces of $B/e$ in $G/e$ and $\ptz(G/e; \calf/e) \leq 1 = \ptz(G; \calf)$. 

Now suppose that for some $k \geq 1$, the result is true for any graph $H$ and set of forces $\mathcal{Q}$ with $\ptz(H; \mathcal{Q}) \leq k$. Again, let $G$ be a graph with standard zero forcing set $B \subseteq V(G)$. Now, suppose $\calf$ is a set of standard forces of $B$ with $\ptz(G; \calf) = k+1$. Let $e = uv$ be a given edge in $G$ such that $(u \rightarrow v) \in \calf$. Define $T(\calf)$ to be all vertices in $G$ that are forced last in $\calf$ (at time step $k+1$). For all vertices $q \in T(\calf)$, let $q'$ be the vertex in $G$ that forces $q$ at time step $k+1$. Note that for any $q \in T(\calf)$ and any neighbor $y$ of $q$ in $G$ with $y \neq q'$, $y$ is also in $T(\calf)$. This is because if $y \notin T(\calf)$, then $y$ cannot perform a force until $q$ is forced. However, $q$ is forced in time step $k+1$ which implies that $y$ forces in a time step greater than $\ptz(G; \calf)$, and this is a contradiction. Suppose $uv = q'q$ for some $q \in T(\calf)$. Since $N(v) \setminus \{u\} \subseteq T(\calf)$, $\calf/e$ is a set of forces of $B/e$ in $G/e$ such that $\ptz(G/e; \calf/e) \leq k+1 = \ptz(G; \calf)$.

Finally, suppose $u \rightarrow v$ in $\calf$ at a time step less than $k+1$. Construct $G/e$ by the following process. First, remove $T(\calf)$ from $G$ to obtain $H = G - T(\calf)$. Next, contract $e$ in $H$ to obtain $H/e$. Finally, add $T(\calf)$ to $H$ so that the neighborhood in $H$ of each $q \in T(\calf)$ is the same as the neighborhood of $q$ in $G$ (except that there may be a $q \in T(\calf)$ such that $v_e \sim q$ in $G/e$ whereas $v \sim q$ in $G$). Let $\calf' = \calf \setminus \{q' \rightarrow q \ | \ q \in T(\calf)\}$. Clearly $\ptz(H; \calf') \leq k$. So by the induction hypothesis, $\ptz(H/e; \calf'/e) \leq \ptz(H; \calf') \leq k$. When $T(\calf)$ is added to $H/e$ and the set of forces $\calf/e$ is considered instead of $\calf'/e$, the propagation time will increase by at most $1$. Thus, $\ptz(G/e; \calf/e) \leq \ptz(H/e; \calf'/e) + 1 \leq k+1 = \ptz(G; \calf).$ Note that if $\calf$ and $B$ are chosen such that $\ptz(G; \calf) = \ptz(G; B)$ and $\thz(G) = \thz(G; B)$, then 
\beas 
\thz(G/e) \leq |B/e| + \ptz(G/e; \calf/e) \leq |B| + \ptz(G; \calf) = |B| + \ptz(G; B) = \thz(G). \qedhere
\eeas
\end{proof}

\begin{thm}\label{thzfChar}
Given a graph $G$ and a positive integer $t$, $\thzf(G) \leq t$ if and only if there exists integers $a \geq 1$ and $b \geq 0$ such that $a+b = t$ and $G$ can be obtained from $K_a \square P_{b+1}$ by contracting path edges and deleting edges.
\end{thm}

\begin{proof}
First suppose $\thzf(G) \leq t$. Let $H$ be a spanning supergraph of $G$ such that $H$ has a standard zero forcing set $B$ with $\thz(G; B) \leq t$. Let $\calf$ be a set of $\Z$ forces of $B$ in $H$ such that $\ptz(H; \calf) = \ptz(H; B)$. Let $a = |B|$, $b' = \ptz(H; B) = \thz(G;B) - a$, and $b = t-a$. Then $b' \leq b$ and
\beas 
G \leq H \preceq \mathcal{E}(H, B, \calf) \leq K_a \square P_{b'+1} \leq K_a \square P_{b+1}.
\eeas 
Note that by the construction of $H$ and $\mathcal{E}(H, B, \calf)$, $H$ can be obtained from $K_a \square P_{b+1}$ by contracting path edges. Then $G$ can be obtained from $H$ by deleting edges.

For the other direction, suppose $G' = K_a \square P_{b+1}$ with $a+b = t$ and $G$ can be obtained from $G'$ by contracting path edges and deleting edges. Choose $B' \subseteq V(G')$ such that $B'$ induces a copy of $K_a$ in $G'$ that corresponds to an endpoint of $P_{b+1}$. Note that $B'$ is a standard zero forcing set of $G'$ with set of forces $\calf'$ such that the set $\{uv \ | \ (u \rightarrow v) \in \calf'\}$ is the set of path edges in $G'$. In other words, $\calf'$ propagates along the path edges of $G'$. Also note that $\ptz(G'; \calf') = b$ and $|B| = a$. Let $D$ be a set of edges and let $C$ be a set of path edges in $G'$ such that $G$ can be obtained from $G'$ by first contracting the edges in $C$, then deleting the edges in $D$. Let $H'$ be the graph obtained from $G'$ by contracting the edges in $C$. Note that $D \subseteq E(H')$. By repeated applications of Lemma $\ref{contractZF}$, it is possible to obtain a standard zero forcing set $B \subseteq V(H')$ with set of forces $\calf$ such $|B| \leq |B'|$ and $\ptz(H'; \calf) \leq  \ptz(G'; \calf') = b$. Thus, 
\beas 
\thzf(H') \leq \thz(H') \leq |B| + \ptz(H'; \calf) \leq |B'| + \ptz(G'; \calf') = a + b = t. 
\eeas
By Theorem \ref{zfThrotSubMon}, $\thzf(G) \leq \thzf(H') \leq t$.
\end{proof}
Note that if a fixed integer $t \geq 1$ is given, then the graphs that have $\zf$ throttling number at most $t$ are exactly the graphs given by Theorem \ref{thzfChar}. The following corollary is immediate from this observation.
\begin{cor}
If $t$ is a fixed positive integer, then there are finitely many graphs with $\zf$ throttling number equal to $t$.
\end{cor}

The next theorem uses Theorem \ref{thzfChar} to show that $\thzf$ does not inherit the property of minor monotonicity from $\zf$. Recall that the maximum degree of a graph $G$ is denoted as $\Delta(G)$.

\begin{thm}\label{thzfNotMinMon}
The $\zf$ throttling number of a graph is not minor monotone.
\end{thm}

\begin{proof}
Consider the graph $K_3 \square P_3$ and let $B \subseteq V(K_3 \square P_3)$ be the three vertices in a copy of $K_3$ that corresponds to an endpoint of $P_3$. Since $\pt_{\zf}(K_3 \square P_3; B) \leq 2$, $\thzf(K_3 \square P_3) \leq 5$. Let $G$ be the minor of $K_3 \square P_3$ shown on the left in Figure \ref{GH}. The following argument shows that $G$ cannot be obtained from $K_a \square P_{b+1}$ with $a + b = 5$ by contracting path edges and/or deleting edges. Since $|V(K_1 \square P_5)| = |V(K_5 \square K_1)| = 5 < 8 =|V(G)|$, $G$ cannot be obtained from $K_1 \square P_5$ or $K_5 \square P_1$ without adding vertices. Note that $|V(K_2 \square P_4)| = |V(K_4 \square P_2)| = 8$ which means that contractions are not allowed in order to obtain $G$ from those graphs. Since $\Delta(K_2 \square P_4) = 3$, $\Delta(K_4 \square P_2) = 4$, and $\Delta(G) = 5$, $G$ cannot be obtained from those graphs by deleting edges. To obtain $G$ from $K_3 \square P_3$ using the operations in Theorem \ref{thzfChar}, exactly one contraction of a path edge is required since $|V(G)| = 8$ and $|V(K_3 \square P_3)| = 9$. Note that by the symmetry of $K_3 \square P_3$, contracting any single path edge yields the same graph. Let $G'$ be the graph obtained by contracting a path edge of $K_3 \square P_3$ shown in the middle of Figure \ref{GH}. The degree sequences of $G'$ and $G$ are $(5,4,4,3,3,3,3,3)$ and $(5,3,3,3,3,3,3,3)$ respectively. Thus, the only possible way to delete edges in $G'$ and obtain $G$ is by deleting the edge between the two vertices of degree 4. Delete this edge from $G'$ and let $H$ be the resulting graph shown on the right in Figure \ref{GH}. If $v_1$ and $v_2$ are the vertices of degree $5$ in $G$ and $H$ respectively, then $H-v_2$ contains a $6$-cycle and $G - v_1$ does not. Therefore, $G$ is not isomorphic to $H$ and $G$ cannot be obtained from $K_a \square P_{b+1}$ with $a + b = 5$ by contracting path edges and/or deleting edges. By Theorem \ref{thzfChar}, this means that $\thzf(G) \geq 6$. Since $\thzf(K_3 \square P_3) \leq 5$, it follows that $\thzf$ is not minor monotone. \qedhere

\end{proof}
\begin{figure}[H] \begin{center}
\scalebox{1}{\includegraphics{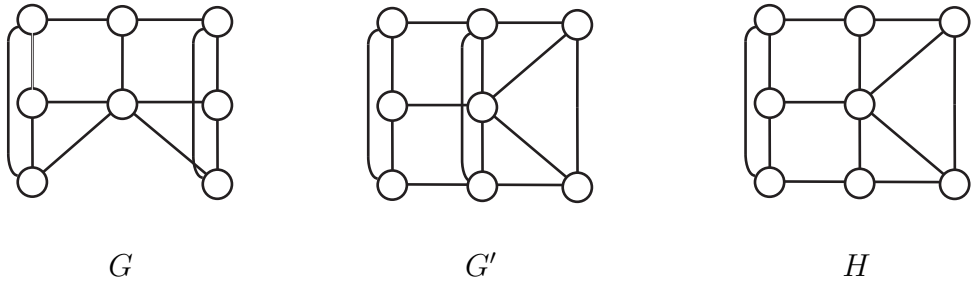}}\\
\end{center}
\hspace{1.2 in} $G$ \hspace{1.63 in} $G'$ \hspace{1.7 in} $H$
\begin{center}
\caption{The graphs $G$, $G'$, and $H$ are minors of $K_3 \square P_3$ used in the proof of Theorem \ref{thzfNotMinMon}.}\label{GH}
\end{center}
\end{figure}

In the next section, the proof of Theorem \ref{thzfChar} is modified in order to characterize standard throttling.
\end{section}

\begin{section}{A Characterization for Standard Throttling}\label{sectStandThrot}
Since there are graphs (e.g., stars) for which $\thz \neq \thzf$, it is clear that the characterization in Theorem \ref{thzfChar} does not also characterize $\thz$. However, the only part of this characterization that does not work for standard throttling is the deletion of edges. In fact, Example \ref{StarWheelEx} demonstrates that standard throttling is not spanning subgraph monotone. The next theorem shows how $\thz$ can be characterized by being more careful about which edges can be deleted. 
\begin{thm}\label{throtChar}
Given a graph $G$ and a positive integer $t$, $\thz(G) \leq t$ if and only if there exists integers $a \geq 0$ and $b \geq 1$ such that $a+b = t$ and $G$ can be obtained from $K_a \square P_{b+1}$ by contracting path edges and deleting complete edges.
\end{thm}

\begin{proof}
First suppose $\thz(G) \leq t$. Let $B \subseteq V(G)$ be a standard zero forcing set of $G$ with $\thz(G; B) \leq t$ and let $\calf$ be a set of standard forces of $B$ in $G$ with $\ptz(G; \calf) = \ptz(G; B)$. Let $a = |B|$, $b' = \ptz(G; B) = \thz(G;B) - a$, and $b = t-a$. Then $b' \leq b$ and
\beas 
G \preceq \mathcal{E}(G, B, \calf) \leq K_a \square P_{b'+1} \leq K_a \square P_{b+1}.
\eeas 
Note that by the construction of $\mathcal{E}(G, B, \calf)$, $G$ can be obtained from $K_a \square P_{b+1}$ by contracting path edges and deleting complete edges.

For the other direction, suppose $G' = K_a \square P_{b+1}$ with $a+b = t$ and $G$ can be obtained from $G'$ by contracting path edges and deleting complete edges. Choose $B' \subseteq V(G')$ such that $B'$ induces a copy of $K_a$ in $G'$ that corresponds to an endpoint of $P_{b+1}$. Note that $B'$ is a standard zero forcing set of $G'$ with set of forces $\calf'$ such that the set $\{uv \ | \ (u \rightarrow v) \in \calf'\}$ is the set of path edges in $G'$. In other words, $\calf'$ propagates along the path edges of $G'$. Also note that $\ptz(G'; \calf') = b$ and $|B'| = a$. Let $D$ be a set of complete edges in  $G'$ and let $C$ be a set of path edges in $G'$ such that $G$ can be obtained from $G'$ by first deleting the edges in $D$, then contracting the edges in $C$. Let $H'$ be the graph obtained from $G'$ by deleting the edges in $D$. Since no edge in $D$ is used to perform a force in $\calf'$, $\calf'$ is still a set of forces of $B'$ in $H'$ with $\ptz(H'; \calf') \leq \ptz(G'; \calf') = b$. Also, $G$ can be obtained from $H'$ by contracting the edges in $C$. By repeated applications of Lemma $\ref{contractZF}$, it is possible to obtain a standard zero forcing set $B \subseteq V(G)$ with set of forces $\calf$ such $|B| \leq |B'|$ and $\ptz(G; \calf) \leq  \ptz(H'; \calf') \leq b$. Thus, 
\beas 
\thz(G) \leq |B| + \ptz(G; \calf) \leq |B'| + \ptz(H'; \calf') \leq |B'| + \ptz(G'; \calf') = a + b = t. \qedhere
\eeas
\end{proof} 
\begin{cor}
If $t$ is a fixed positive integer, then there are finitely many graphs $G$ with standard throttling number equal to $t$.
\end{cor}

Suppose $G$ is a graph on $n$ vertices and $t$ is a postive integer with $\thz(G) \leq t$. Note that $t$ can be used to bound the number of vertices in $G$. Since $\ceil{2\sqrt{n} - 1} \leq \thz(G) \leq t$, $|V(G)| = n \leq \frac{(t+1)^2}{4}$. By Corollary \ref{thzfLowerBd}, this bound still holds when $\thzf(G) \leq t$.

In order to construct forcing sets in paths and cycles that are optimal for throttling, it has been useful to ``snake'' the graph in some way. This idea was used for $\thz(P_n)$ in \cite{BY13Throt}, and again for $\thz(C_n)$ in \cite{PSD}. A ``snaking'' construction was also used for $\thzf(C_n)$ in Proposition \ref{zfthrotCycle} (see Figure \ref{CycleSnake}). Note that in most of these cases, the ``snaked'' graph is a spanning subgraph or a minor of a graph of the form $K_a \square P_{b+1}$. It is interesting to observe that the ``snaking" method is present in Theorems \ref{thzfChar} and \ref{throtChar}.

\end{section}

\begin{section}{Extreme Throttling}\label{sectApps}
This section uses Theorems \ref{thzfChar} and \ref{throtChar} to quickly characterize graphs with low throttling numbers. The connection between $\thzf$ and the independence number of a graph is also investigated. This connection is used to give a necessary condition for graphs $G$ with $\thzf(G) = n$.

For a fixed positive integer $t$, Theorem \ref{thzfChar} characterizes all graphs $G$ with $\thzf(G) \leq t$. Clearly $\thzf(G) = t$ if and only if $\thzf(G) \leq t$ and $\thzf(G) \nleq t-1$. So all graphs with $\thzf(G) = t$ can be characterized by applying Theorem \ref{thzfChar} and removing the graphs with $\zf$ throttling number at most $t-1$. This is done by hand for $t \leq 3$ as follows.

\begin{obs}
The graph $G = K_1$ is the only graph with $\thzf(G) = 1$.
\end{obs}

\begin{prop}
For a graph $G$, $\thzf(G) = 2$ if and only if $G = K_2$ or $G = 2K_1$.
\end{prop}
\begin{proof}
By Theorem \ref{thzfChar}, $\thzf(G) \leq 2$ if and only if $G$ can be obtained from $K_1 \square P_2 = K_2$ or $K_2 \square P_1 = K_2$ by deleting edges and contracting path edges. Thus, $\thzf(G) \leq 2$ if and only if $G \in \{K_1, K_2, 2K_1\}$. Since $G = K_1$ is the only graph that satisfies $\thzf(G) = 1$, $\thzf(G) = 2$ if and only if $G \in \{K_2, 2K_1\}$.
\end{proof}
\begin{prop}
For a graph $G$, $\thzf(G) = 3$ if and only if $G \in \mathcal{G}$ where 
\beas 
\mathcal{G} = \{C_4, P_4, 2K_2, K_1 \dot{\cup} P_3, K_2 \dot{\cup} 2K_1, 4K_1, K_3, P_3, K_1 \dot{\cup} K_2, 3K_1\}.
\eeas 
\end{prop}
\begin{proof}
By Theorem \ref{thzfChar}, $\thzf(G) \leq 3$ if and only if $G$ can be obtained from $K_3 \square P_1 = K_3$, $K_2 \square P_2 = C_4$, or $K_1 \square P_3 = P_3$ by deleting edges and contracting path edges. Let $\mathcal{H}$ be the set of all subgraphs of $C_4$ and $K_3$. It is clear that $\thzf(G) \leq 3$ if and only if $G \in \mathcal{H}$. Note that $\mathcal{G} = \mathcal{H} \setminus \{K_1, K_2, 2K_1\}$.
\end{proof}

Theorems \ref{thzfChar} and \ref{throtChar} reinforce the fact that for any graph $G$, $\thzf(G) \leq \thz(G)$. Let $G$ be a graph. Since $\thz$ is bounded below by $\thzf$, if there is a subset $B \subseteq V(G)$ with $\thz(G; B) = \thzf(G)$,  then $\thz(G) = \thzf(G)$. 

\begin{cor}
If $t \in \{1, 2, 3\}$ and $G \notin \{K_1 \dot{\cup} P_3, K_2 \dot{\cup} 2K_1, 4K_1\}$, then $\thz(G) = t$ if and only if $\thzf(G) = t$.
\end{cor}
\begin{proof}
Let $\mathcal{J} = \{K_1 \dot{\cup} P_3, K_2 \dot{\cup} 2K_1, 4K_1\}$. For each graph $G$ with $\thzf(G) \leq 3$ and $G \notin \mathcal{J}$, it is possible to produce a standard zero forcing set $B \subseteq V(G)$ with $\thz(G; B) = \thzf(G)$. If $G \in \mathcal{J}$, then $\thzf(G) = 3$, but $\thz(G) = 4$ because forcing by a hop is no longer allowed.
\end{proof}

High $\zf$ throttling values are harder to characterize. Clearly, $\thzf(K_n) = \thz(K_n) = n$. Let $(K_n)_e$ be the complete graph on $n$ vertices minus a single edge. It is also clear that $\thzf((K_n)_e) = \thz((K_n)_e) = n$. More generally, $\thzf(G) = n$ implies that $\thz(G) = n$. For a given graph $G$, the following proposition gives an upper bound for $\thzf(G)$ in terms of the independence number, $\alpha(G)$. 

\begin{prop}\label{alphabound}
If $G$ is a graph of order $n$, then $\thzf(G) \leq n - \alpha(G) + \ceil{2 \sqrt{\alpha(G)} - 1}$.
\end{prop}
\begin{proof}
Suppose $G$ is a graph with independent set $A \subseteq V(G)$. Let $B = V(G) \setminus A$. Note that $G - B$ has no edges and by Theorem \ref{zfThrotSubMon}, $\thzf(G-B) \leq \thzf(C_{|A|}) = \ceil{2 \sqrt{|A|} - 1}$. Choose $C \subseteq A$ such that $\thzf(G-B, C) = \ceil{2 \sqrt{|A|} - 1}$. Then $B \cup C$ is a $\zf$ forcing set of $G$ with $\ptzf(G; B \cup C) \leq \ptzf(G-B, C)$. Thus, $\thzf(G) \leq n - |A| + \ceil{2 \sqrt{|A|} - 1}$. If $A$ satisfies $|A| = \alpha(G)$, the desired result is obtained.
\end{proof}

Since $\alpha(K_{1, n-1}) = n-1$, Example \ref{StarWheelEx} shows that the bound in Proposition \ref{alphabound} is tight.

\begin{cor}\label{alphaCor}
If $G$ is a graph with $\thzf(G) = n$, then $\alpha(G) \leq 3$.
\end{cor}
\begin{proof}
Let $G$ be a graph and define $f(x) = x - \ceil{2 \sqrt{x} - 1}$. So Proposition \ref{alphabound} says that $\thzf(G) \leq n - f(\alpha(G))$. If $x \geq 4$ is an integer, then $f(x) \geq 1$. So if $\alpha(G) \geq 4$, then $\thzf(G) \leq n - f(\alpha(G)) \leq n - 1$. 
\end{proof}
Note that the converse of Corollary \ref{alphaCor} is false. For example, let $G = P_6$. Then $\alpha(G) = 3$ and $\thzf(G) = \ceil{2\sqrt{6} - 1} = 4 < 6 = n$.

\end{section}

\begin{section}{Concluding Remarks}\label{conclusion}
For a graph $G$ and an integer $k$, define the \textsc{Z Floor Throttling} problem as the decision problem of determining whether $\thzf(G) < k$. The complexity of \textsc{Z Floor Throttling} is an interesting question. Recall that for two graphs $G_1$ and $G_2$, the graph $G_1 \dot \cup G_2$ has vertex set and edge set equal to $V(G_1) \dot \cup V(G_2)$ and $E(G_1) \dot \cup E(G_2)$ respectively. For any graph $G$, let $X(G)$ be the set of subsets of $V(G)$ that are $\zf$ forcing sets of $G$. A list of conditions is given in \cite[Theorem 1]{powerdomthrot} that would guarantee that \textsc{Z Floor Throttling} is NP-Complete. One of these conditions is that $X(G_1 \dot \cup G_2) = \{S_1 \dot \cup S_2 \ | \ S_1 \in X(G_1) \text{ and }S_2 \in X(G_2)\}$ for any two graphs $G_1$ and $G_2$. Due to hopping, this condition is not satisfied for $\zf$ forcing sets. For example, let $G_1$ and $G_2$ each be the graph consisting of a single vertex labeled $v_1$ and $v_2$ respectively. Let $S_1 = \emptyset$ and $S_2 = \{v_2\}$. Note that $S_1 \dot \cup S_2$ is a $\zf$ forcing set of $G_1 \dot \cup G_2$ since $v_2$ can force $v_1$ by a hop. However, $S_1$ is not a $\zf$ forcing set of $G_1$. So the conditions given in \cite[Theorem 1]{powerdomthrot} cannot be used to prove that \textsc{Z Floor Throttling} is NP-Complete. It would be useful to have other tools to help determine the complexity of the \textsc{Z Floor Throttling} problem.

Corollary \ref{alphaCor} states that a low independence number is necessary in order to achieve a maximum $\zf$ throttling number. Another possible direction for future work is to completely characterize high $\zf$ throttling numbers. It would also be interesting to determine the exact relationship between $\alpha$ and $\thzf$. It is noted in \cite[Remark 2.47]{Parameters} that for any graph $G$, $\floor{Z_{\ell}}(G) \leq \zf(G) \leq \floor{Z_{\ell}}(G) + 1$. This motivates a comparison of $\thzf$ and $\operatorname{th}_{\floor{\Z_{\ell}}}$. If $\thzf$ and $\operatorname{th}_{\floor{\Z_{\ell}}}$ can be arbitrarily far apart, then studying $\floor{\Z_{\ell}}$ throttling may be of interest on its own. 
\end{section}

\section*{Acknowledgment}

Research supported in part by Holl Chair funds.

\end{document}